\documentclass{compositio}
\usepackage{mathrsfs}
\usepackage{amsmath}
\usepackage{amssymb}
\usepackage{verbatim}
\usepackage{xcolor}
\usepackage{graphicx}
\usepackage[all]{xy}
\numberwithin{equation}{section}

\theoremstyle{plain}
\newtheorem{theorem}{Theorem}[section]
\newtheorem{proposition}{Proposition}[section]
\newtheorem{corollary}{Corollary}[section]
\newtheorem{lemma}{Lemma}[section]

\theoremstyle{definition}
\newtheorem{definition}{Definition}[section]

\theoremstyle{remark}
\newtheorem{rem}{Remark}[section]

\begin{document}

\title{Nefness of the direct images of pluricanonical bundles}

\author{Jingcao Wu}
\email{jingcaowu13@fudan.edu.cn}
\address{School of Mathematical Sciences, Fudan University, Shanghai 200433, People's Republic of China}

\thanks{This research was supported by China Postdoctoral Science Foundation, grant 2019M661328.}

\begin{abstract}
Given a fibration $f$ between two projective manifolds $X$ and $Y$, we provide a sufficient condition such that the direct images $f_{\ast}(K_{X/Y}\otimes L\otimes\mathscr{I}(f,\|L\|))$ is nef, where $L$ is a holomorphic line bundle with non-negative relative Iitaka dimension and $\mathscr{I}(f,\|L\|)$ is the relative asymptotic multiplier ideal sheaf. 
\end{abstract}

\maketitle

\section{Introduction}
\label{sec:introduction}

Assume that $f:X\rightarrow Y$ is a fibration i.e. a surjective morphism with connected fibres between two projective manifolds $X$ and $Y$, and $K^{m}_{X/Y}$ is the $m$-th tensor power of the relative pluricanonical bundle on $X$. The positivity of the associated direct image $f_{\ast}(K^{m}_{X/Y})$ is of significant importance for understanding the geometry of this fibration. Fruitful results have been published on this subject, for example \cite{Ber08,Ber09,Hor10,Kaw81,Kaw82,Ko86a,Ko86b,Ko87,Vie82b,Vie83}. It turns out that the positivity of $f_{\ast}(K^{m}_{X/Y})$ is deeply influenced by $K_{X/Y}$ when $m>1$.

In this paper, we focus upon the situation that $K_{X/Y}$ has non-negative relative Iitaka dimension $\kappa(K_{X/Y},f)$ (see Sect.\ref{sec:iitaka}); our main theorem is as follows:

\begin{theorem}\label{t11}
Let $f:X\rightarrow Y$ be a smooth fibration between projective manifolds $X$ and $Y$. Fix $p\gg0$ and divisible enough that computes $\mathscr{I}(f,\|K_{X/Y}\|)$ (see Sect.\ref{sec:asymptotic}). For any non-negative integer $k$, if $\mathscr{I}(kp\|K_{X/Y}\|)=\mathcal{O}_{X}$, 
\[
f_{\ast}(K^{m}_{X/Y})
\]
is nef for any non-negative integer $m\leqslant k+1$.
\end{theorem}

Here, and throughout the rest of this paper, $\mathscr{I}(\|K_{X/Y}\|)$ (resp. $\mathscr{I}(f,\|K_{X/Y}\|)$) refers to the (resp. relative) asymptotic multiplier ideal sheaf (see Sect.\ref{sec:asymptotic}). Instead of Theorem \ref{t11}, we would like to arrange all the things for the following refined version that is formulated for an arbitrary line bundle $L$ on $X$. Theorem \ref{t11} is then a quick consequence of it.

\begin{theorem}\label{t12}
Let $f:X\rightarrow Y$ be a smooth fibration between projective manifolds $X$ and $Y$. Let $L$ be a holomorphic line bundle on $X$ with $\kappa(L,f)\geqslant0$ and $\kappa(K_{X/Y}\otimes L,f)\geqslant0$. Fix $p\gg0$ and divisible enough that computes $\mathscr{I}(f,\|L\|)$. Assume that the following conditions hold:
\begin{enumerate}
\item[(1)] there exists a (singular) metric $\tau$ on $-K_{X/Y}\otimes L^{p}$ such that $i\Theta_{-K_{X/Y}\otimes L^{p},\tau}\geqslant0$ and $\mathscr{I}(\tau)=\mathcal{O}_{X}$; and
\item[(2)] let $\varphi$ and $\psi$ be the metrics on $L$ and $K_{X/Y}\otimes L$ associated to (see Sect.\ref{sec:asymptotic})
\[
\mathscr{I}(f,\|L\|)\textrm{ and }\mathscr{I}(f,\|K_{X/Y}\otimes L\|)
\] 
respectively, then $\varphi$ is less singular than $\psi$ (see \cite{Dem12}), i.e.
$\varphi\preceq\psi$.
\end{enumerate}
Then
\[
\mathcal{E}:=f_{\ast}(K_{X/Y}\otimes L\otimes\mathscr{I}(f,\|L\|))
\]
is nef.
\end{theorem}

For a line bundle $H$, $H^{k}$ refers to its $k$-th tensor power with the convention that $H^{0}=\mathcal{O}_{X}$ and $H^{k}=(H^{\ast})^{-k}$ for $k<0$. $-K_{X/Y}$ is the relative anti-canonical bundle. Remember in \cite{Wu21}, Theorem 6.1 we show that $\mathcal{E}$ is always torsion-free, hence the conclusion of Theorem \ref{t12} is interpreted as that $\mathcal{E}$ is nef as a torsion-free coherent sheaf (see Sect.\ref{sec:positivity}). 

The requirements (1), (2) are verified in other situations (see Corollaries \ref{c41} and \ref{c42}) apart from Theorem \ref{t11}. 

We use the same strategy as in \cite{Ko86a,Ko86b,Vie82b,Vie83} to prove Theorem \ref{t12}. The ideal is expanded as follows: first we should prove an injectivity theorem.

\begin{theorem}\label{t13}
Let $f:X\rightarrow Y$ be a fibration between projective manifolds $X$ and $Y$. Let $L$ be a line bundle on $X$ with $\kappa(L,f)\geqslant0$, and let $A^{\prime}$ be the pullback of a sufficient ample line bundle $A$ on $Y$. Fix $p\gg0$ and divisible enough that computes $\mathscr{I}(f,\|L\|)$. Assume that there exists a (singular) metric $\tau$ on $-K_{X/Y}\otimes L^{p}$ such that $i\Theta_{-K_{X/Y}\otimes L^{p},\tau}\geqslant0$ and $\mathscr{I}(\tau)=\mathcal{O}_{X}$. 

Then for a (non-zero) section $s$ of $A^{\prime}$, the multiplication map induced by the tensor product with $s$
\[
\Phi:H^{q}(X,K_{X}\otimes L\otimes A^{\prime}\otimes\mathscr{I}(f,\|L\|))\rightarrow H^{q}(X,K_{X}\otimes L\otimes (A^{\prime})^{2}\otimes\mathscr{I}(f,\|L\|))
\]
is well-defined and injective for any $q\geqslant0$. 
\end{theorem}
Combine with Theorem \ref{t13} and the fact that $\mathcal{E}$ is torsion-free, we obtain the following Koll\'{a}r-type vanishing theorem.

\begin{theorem}\label{t14}
Let $f:X\rightarrow Y$ be a fibration between projective manifolds $X$ and $Y$. Let $L$ be a holomorphic line bundle on $X$ with $\kappa(L,f)\geqslant0$. Fix $p\gg0$ and divisible enough that computes $\mathscr{I}(f,\|L\|)$. Assume that there exists a (singular) metric $\tau$ on $-K_{X/Y}\otimes L^{p}$ such that $i\Theta_{-K_{X/Y}\otimes L^{p},\tau}\geqslant0$ and $\mathscr{I}(\tau)=\mathcal{O}_{X}$. If $A$ is a sufficient ample line bundle (but independent of $f$ and $L$) on $Y$, then for any $i>0$ and $q\geqslant0$,
\[
H^{i}(Y,R^{q}f_{\ast}(K_{X}\otimes L\otimes\mathscr{I}(f,\|L\|))\otimes A)=0.
\] 
\end{theorem}
Here we emphasize that the choice of $A$ in both of the theorems above is independent of $f$ and $L$. In particular, Theorem \ref{t14} is not an easy consequence of asymptotic Serre vanishing theorem \cite{Har77}.

A direct consequence of Theorem \ref{t14} is
\begin{corollary}\label{c11}
Under the same assumptions as in Theorem \ref{t14}, if $A$ is moreover globally generated and $A^{\prime}$ is a nef line bundle on $Y$, then the sheaf
\[
R^{q}f_{\ast}(K_{X}\otimes L\otimes\mathscr{I}(f,\|L\|))\otimes A^{m}\otimes A^{\prime}
\] 
is globally generated for any $q\geqslant0$ and $m\geqslant\dim Y+1$.
\end{corollary}

Then we prove that $\mathcal{E}$ is weakly positive in the sense of Viehweg (see Sect.\ref{sec:positivity}) via the fibre product method of Viehweg \cite{Vie82b,Vie83}, and Theorem \ref{t12} follows. We leave the details in the text.

This paper is organised as follows. We first recall some background materials, including the definition of the positivity concerning a torsion-free coherent sheaf, the asymptotic multiplier ideal sheaf and so on. Then, we proceed to prove Theorems \ref{t13} and \ref{t14} in Sect.\ref{sec:injectivity}. The proof of Theorems \ref{t11} and \ref{t12} are presented in Sect.\ref{sec:direct}.

\section{Preliminary}
\label{sec:preliminary}
In this section we will introduce some basic materials. For clarity and for convenience of later reference, it will be done in the following setting: $f:X\rightarrow Y$ is a fibration between two projective manifolds, and $L$ is a holomorphic line bundle on $X$ with $\kappa(L,f)\geqslant0$. Here $\kappa(L,f)$ is the relative Iitaka dimension that is explained immediately.

\subsection{Relative Iitaka dimension}
\label{sec:iitaka}
This part is borrowed from \cite{FEM13}.

Let $l$ be the dimension of a general fibre $F$ of $f$. We have
\begin{proposition}\label{p21}
For every coherent sheaf $\mathcal{G}$ on $X$, there is $C>0$ (independent of $L$) such that
\[
\textrm{rank}(f_{\ast}(\mathcal{G}\otimes L^{k}))\leqslant Ck^{l} \textrm{ for all }k\gg0.
\]
\begin{proof}
Let us write $L=A\otimes B^{-1}$, with $A$ and $B$ are very ample line bundles. For every $k$, if we choose $E$ general in the complete linear system $|B^{k}|$, then a
local defining function of $E$ is a non-zero divisor on $\mathcal{G}$, in which case we have an inclusion
\[
H^{0}(F,\mathcal{G}\otimes L^{k})\hookrightarrow H^{0}(F,\mathcal{G}\otimes A^{k}).
\]
Since $A$ is very ample, we know that there is a polynomial $P\in\mathbb{Q}[t]$ \cite{Laz04} with $\deg(P)\leqslant l$ such that $h^{0}(F,\mathcal{G}\otimes A^{k})=P(k)$ for $k\gg0$. Therefore $h^{0}(F,\mathcal{G}\otimes L^{k})\leqslant P(k)\leqslant Ck^{l}$ for a suitable $C>0$ and all $k\gg0$.
\end{proof}
\end{proposition}

\begin{definition}\label{d21}
The relative Iitaka dimension $\kappa(L,f)$ of $L$ is the biggest integer $M$ such that there is $C>0$ satisfying
\[
\textrm{rank}(f_{\ast}L^{k})\geqslant Ck^{M} \textrm{ for all }k\gg0
\]
with the convention that $\kappa(L,f)=-\infty$ if $\textrm{rank }f_{\ast}L^{k}=0$.
\end{definition}
Note that $\kappa(L,f)$ takes value in $\{-\infty, 0,1,...,l\}$ by Proposition \ref{p21}. In particular, if $\kappa(L,f)=l$, we say that $L$ is $f$-big.

\subsection{Positivity}
\label{sec:positivity}
Let $E$ be a holomorphic vector bundle of rank $r$ over $X$. By $\mathbb{P}(E^{\ast})$, we denote the projectivised bundle of $E^{\ast}$ and by $\mathcal{O}_{E}(1):=\mathcal{O}_{\mathbb{P}(E^{\ast})}(1)$ we denote the tautological line bundle. Let
\[
\pi:\mathbb{P}(E^{\ast})\rightarrow X
\]
be the canonical projection.

We now collect the definitions of positivity from \cite{DPS94,DPS01,Pau16,PT14,Vie82b,Vie83} as follows.

\begin{definition}\label{d22}
\begin{enumerate}
  \item $E$ is weakly positive in the sense of Viehweg if, on some Zariski open subset $U\subset X$, for any integer $a>0$, there exists an integer $b>0$ such that $S^{ab}(E)\otimes bA$ is generated by global sections on $U$. Here, $A$ is an auxiliary ample line bundle over $X$ and $S^{ab}(E)$ refers to the $ab$-th symmetric product of $E$.
  \item $E$ is pseudo-effective if $\mathcal{O}_{E}(1)$ is pseudo-effective and the image of the non-nef locus 
  \[
  \text{NNef}(\mathcal{O}_{E}(1))
  \] 
  (i.e., the union of all curves $C$ on $\mathbb{P}(E^{\ast})$ with $\mathcal{O}_{E}(1)\cdot C<0$)   under $\pi$ is a proper subset of $X$.
  \item $E$ is nef if $\mathcal{O}_{E}(1)$ is nef.
  \item $E$ is almost nef if there exists a countable family of proper subvarieties $Z_{i}$ of $X$ such that $E|_{C}$ is nef for any curve $C\not\subset\cup_{i}Z_{i}$.
  \item $E$ is ample if $\mathcal{O}_{E}(1)$ is ample.
\end{enumerate}
\end{definition}

\begin{rem}\label{r21}
The relationships among these notions are summarised below.

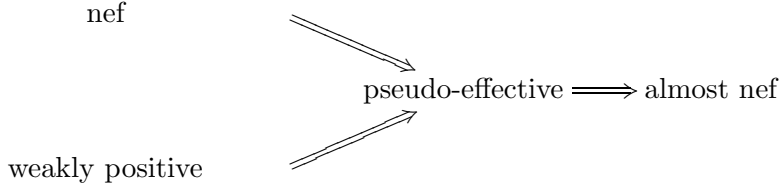
\xymatrix@R=0.5cm{
\textrm{nef} &  \ar@{=>}[dr]  &    &                       \\
 &  & \textrm{pseudo-effective} \ar@{=>}[r] & \textrm{almost nef}\\
\textrm{weakly positive} & \ar@{=>}[ur]  &    &}
\end{rem}

These notions extend to a torsion-free coherent sheaf. Assume that $\mathcal{E}$ is a torsion-free coherent sheaf; then it is locally free outside of a 2-codimensional subvariety $Z$. We say that $\mathcal{E}$ is nef (resp. pseudo-effective, almost nef,...) if $\mathcal{E}|_{X\setminus Z}$ has the corresponding property. The reader can refer to \cite{Pau16,PT14} for the more details.

\subsection{The asymptotic multiplier ideal sheaf}
\label{sec:asymptotic}
This part is mostly collected from \cite{Laz04b}.

Recall that for an arbitrary ideal sheaf $\mathfrak{a}\subset\mathcal{O}_{X}$, the associated multiplier ideal sheaf is defined as follows: let $\mu:\tilde{X}\rightarrow X$ be a smooth modification such that $\mu^{\ast}\mathfrak{a}=\mathcal{O}_{\tilde{X}}(-E)$, where $E$ has simple normal crossing support. Then given a positive real number $c>0$ the multiplier ideal sheaf is defined as
\[
\mathscr{I}(c\cdot\mathfrak{a}):=\mu_{\ast}\mathcal{O}_{\tilde{X}}(K_{\tilde{X}/X}-\lfloor cE\rfloor).
\]
Here $\lfloor E\rfloor$ means the round-down.

Now assume that $\kappa(L)\geqslant0$. Fix a positive real number $c>0$. For $k>0$ consider the complete linear series $|L^{k}|$, and form the multiplier ideal sheaf
\[
\mathscr{I}(\frac{c}{k}|L^{k}|)\subseteq\mathcal{O}_{X},
\]
where $\mathscr{I}(\frac{c}{k}|L^{k}|):=\mathscr{I}(\frac{c}{k}\cdot\mathfrak{a}_{k})$ with $\mathfrak{a}_{k}$ being the base-ideal of $|L^{k}|$. It is not hard to verify that for every integer $p\geqslant1$ one has the inclusion
\[
\mathscr{I}(\frac{c}{k}|L^{k}|)\subseteq\mathscr{I}(\frac{c}{pk}|L^{pk}|).
\]
Therefore the family of ideals
\[
\{\mathscr{I}(\frac{c}{k}|L^{k}|)\}_{(k\geqslant0)}
\]
has a unique maximal element from the ascending chain condition on ideals.

\begin{definition}\label{d23}
The asymptotic multiplier ideal sheaf associated to $c$ and $|L|$,
\[
\mathscr{I}(c\|L\|)
\]
is defined to be the unique maximal member among the family of ideals $\{\mathscr{I}(\frac{c}{k}|L^{k}|)\}$.
\end{definition}

By definition, $\mathscr{I}(c\|L\|)=\mathscr{I}(\frac{c}{k}|L^{k}|)$ for some $k$. Let $u_{1},...,u_{m}$ be a basis of $H^{0}(X,L^{k})$, then the base-ideal of $|L^{k}|$ is just $\mathcal{I}(u_{1},...,u_{m})$. Let $\varphi=\log(|u_{1}|^{2}+\cdots+|u_{m}|^{2})$, which is a singular metric on $L^{k}$. We verify that
\[
\mathscr{I}(\frac{c}{k}|L^{k}|)=\mathscr{I}(\frac{c}{k}\varphi).
\]
Indeed, let $\mu:\tilde{X}\rightarrow X$ be the smooth modification such that $\mu^{\ast}\mathcal{I}(u_{1},...,u_{m})=\mathcal{O}_{\tilde{X}}(-E)$, where $E$ has simple normal crossing support. Then it is computed in \cite{Dem12} that
\[
\mathscr{I}(\frac{c}{k}\varphi)=\mu_{\ast}\mathcal{O}_{\tilde{X}}(K_{\tilde{X}/X}-\lfloor \frac{c}{k}E\rfloor),
\]
which coincides with the definition of $\mathscr{I}(\frac{c}{k}|L^{k}|)$. In summary, we have
\[
\mathscr{I}(c\|L\|)=\mathscr{I}(\frac{c}{k}\varphi),
\]
and $\frac{1}{k}\varphi$ is called the singular metric on $L$ associated to $\mathscr{I}(\|L\|)$.

Next, we introduce the relative variant. Let $f:X\rightarrow Y$ be a surjective morphism between projective manifolds, and $L$ a line bundle on $X$ whose restriction to a general fibre of $f$ has non-negative Iitaka dimension. Then there is a naturally defined homomorphism
\[
\rho:f^{\ast}f_{\ast}L\rightarrow L.
\]
Let $\mu:\tilde{X}\rightarrow X$ be a smooth modification of $|L|$ with respect to $f$, having the property that the image of the induced homomorphism
\[
\mu^{\ast}\rho:\mu^{\ast}f^{\ast}f_{\ast}L\rightarrow \mu^{\ast}L
\]
is the subsheaf $\mu^{\ast}L\otimes\mathcal{O}_{\tilde{X}}(-E)$ of $\mu^{\ast}L$, $E$ being an effective divisor on $\tilde{X}$ such that $E+\textrm{except}(\mu)$ has simple normal crossing support. Here $\textrm{except}(\mu)$ is the exceptional divisor of $\mu$. Given $c>0$ we define
\[
\mathscr{I}(f,c|L|)=\mu_{\ast}\mathcal{O}_{\tilde{X}}(K_{\tilde{X}/X}-\lfloor cE\rfloor).
\]
Similarly, $\{\mathscr{I}(f,\frac{c}{k}|L^{k}|)\}_{(k\geqslant0)}$ has a unique maximal element.

\begin{definition}\label{d24}
The relative asymptotic multiplier ideal sheaf associated to $f$, $c$ and $|L|$,
\[
\mathscr{I}(f,c\|L\|)
\]
is defined to be the unique maximal member among the family of ideals $\{\mathscr{I}(f,\frac{c}{k}|L^{k}|)\}$.
\end{definition}

By definition, $\mathscr{I}(f,c\|L\|)=\mathscr{I}(f,\frac{c}{k}|L^{k}|)$ for some $k$. Let $\rho_{k}$ be the naturally defined homomorphism
\[
\rho_{k}:f^{\ast}f_{\ast}L^{k}\rightarrow L^{k}.
\]
Let $\mu:\tilde{X}\rightarrow X$ be the smooth modification of $|L^{k}|$ with respect to $f$ such that 
\[
\textrm{Im}(\mu^{\ast}\rho_{k})=\mu^{\ast}L^{k}\otimes\mathcal{O}_{\tilde{X}}(-E).
\] 
Consider $\mu_{\ast}\mathcal{O}_{\tilde{X}}(-E)$ which is an ideal sheaf on $X$. Pick a local coordinate ball $U$ of $Y$, and let $u_{1},...,u_{m}$ be the generators of $\mu_{\ast}\mathcal{O}_{\tilde{X}}(-E)$ on $f^{-1}(U)$. The existence of these generators is obvious concerning the fact that $\textrm{Im}(\mu^{\ast}\rho_{k})=\mu^{\ast}L^{k}\otimes\mathcal{O}_{\tilde{X}}(-E)$. Moreover they can be seen as the sections of $\Gamma(f^{-1}(U),L^{k})$.

Now let $\varphi_{U}=\log(|u_{1}|^{2}+\cdots+|u_{m}|^{2})$, which is a singular metric on $L^{k}|_{f^{-1}(U)}$. It is then easy to verify that
\[
\mathscr{I}(\frac{c}{k}\varphi_{U})=\mathscr{I}(f,\frac{c}{k}|L^{k}|)\textrm{ on }f^{-1}(U).
\]
Furthermore, if $v_{1},...,v_{m}$ are alternative generators and $\psi_{U}=\log(|v_{1}|^{2}+\cdots+|v_{m}|^{2})$, obviously we have $\mathscr{I}(\frac{c}{k}\varphi_{U})=\mathscr{I}(\frac{c}{k}\psi_{U})$. Hence all the $\mathscr{I}(\frac{c}{k}\varphi_{U})$ patch together to give a globally defined multiplier ideal sheaf $\mathscr{I}(\frac{c}{k}\varphi)$ such that
\[
\mathscr{I}(\frac{c}{k}\varphi)=\mathscr{I}(f,\frac{c}{k}|L^{k}|)=\mathscr{I}(f,c\|L\|)\textrm{ on }X.
\]
One should be careful that $\{\frac{1}{k}\varphi_{U}\}$ won't give a globally defined metric on $L$ in general. Hence $\frac{1}{k}\varphi$ is interpreted as the collection of functions $\{\frac{1}{k}\varphi_{U}\}$ by abusing the notation, which is called the collection of (local) singular metrics on $L$ associated to $\mathscr{I}(f,c\|L\|)$.

Now we collect some elementary properties from \cite{Laz04b}. Recall that the relative base-ideal of $|L|$ is by definition the image of the homomorphism
\[
f^{\ast}f_{\ast}L\otimes L^{-1}\rightarrow\mathcal{O}_{X}
\]
determined by $\rho$.
\begin{proposition}\label{p22}
Let $f:X\rightarrow Y$ be a surjective morphism between projective manifolds, and $H_{1}, H_{2}$ are line bundles on $X$ with non-negative relative Iitaka dimension. $k$ and $m$ are arbitrary positive integers. Let $L$ be a line bundle on $X$ with non-negative Iitaka dimension.
\begin{enumerate}
\item Let $\mathfrak{a}_{k,f}=\mathfrak{a}(f,|H^{k}_{1}|)$ be the base-ideal of $|H^{k}_{1}|$ relative to $f$. There exits an integer $k_{0}$ such that for every $k\geqslant k_{0}$, the canonical map $\rho_{k}:f^{\ast}f_{\ast}H^{k}_{1}\rightarrow H^{k}_{1}$ factors through the inclusion $H^{k}_{1}\otimes\mathscr{I}(f,\|H^{k}_{1}\|)$, i.e.
\[
\mathfrak{a}_{k,f}\subseteq\mathscr{I}(f,\|H^{k}_{1}\|).
\]
Equivalently, the natural map
\[
f_{\ast}(H^{k}_{1}\otimes\mathscr{I}(f,\|H^{k}_{1}\|))\rightarrow f_{\ast}(H^{k}_{1})
\]
is an isomorphism.
\item $\mathfrak{a}_{m,f}\cdot\mathscr{I}(f,\|H^{k}_{2}\|)\subseteq\mathscr{I}(f,\|H^{m}_{1}\otimes H^{k}_{2}\|)$.
\item $\mathscr{I}(f,\|H^{k}_{1}\|)\supseteq\mathscr{I}(f,\|H^{k+1}_{1}\|)$ for every $k$.
\item $\mathscr{I}(\|L\|)\subseteq\mathscr{I}(f,\|L\|)$.
\end{enumerate}
\begin{proof}
(i) is proved in \cite{Laz04b}, Proposition 11.2.15.

(ii) Fix $p\gg0$ and divisible enough that computes all of the multiplier ideals $\mathscr{I}(f,\|H^{m}_{1}\|)$, $\mathscr{I}(f,\|H^{k}_{2}\|)$ and $\mathscr{I}(f,\|H^{m}_{1}\otimes H^{k}_{2}\|)$. Let $\mathfrak{b}_{k,f}$ be the base-ideal of $|H^{k}_{2}|$ relative to $f$, and let $\mathfrak{c}_{m,k,f}$ be the base-ideal of $|H^{m}_{1}\otimes H^{k}_{2}|$ relative to $f$. Let $\mu:\tilde{X}\rightarrow X$ be the smooth modification of $\mathfrak{a}_{m,f}$, $\mathfrak{a}_{pm,f}$, $\mathfrak{b}_{pk,f}$ and $\mathfrak{c}_{pm,pk,f}$, such that
\[
\begin{split}
 &\mu^{\ast}\mathfrak{a}_{m,f}=\mathcal{O}_{\tilde{X}}(-E), \mu^{\ast}\mathfrak{a}_{pm,f}=\mathcal{O}_{\tilde{X}}(-F), \\
  &\mu^{\ast}\mathfrak{b}_{pk,f}=\mathcal{O}_{\tilde{X}}(-G) \textrm{ and } \mu^{\ast}\mathfrak{c}_{pm,pk,f}=\mathcal{O}_{\tilde{X}}(-H),
\end{split}
\]
where $E=\sum a_{i}E_{i}$, $F=\sum b_{i}E_{i}$, $G=\sum c_{i}E_{i}$ and $H=\sum d_{i}E_{i}$ have simple normal crossing support. Then for every $i$,
\[
d_{i}\leqslant b_{i}+c_{i}\leqslant pa_{i}+c_{i}
\]
and consequently
\[
-a_{i}-\lfloor\frac{c_{i}}{p}\rfloor\leqslant -\lfloor\frac{d_{i}}{p}\rfloor.
\]
Thus
\[
\begin{split}
\mathfrak{a}_{m,f}\cdot\mathscr{I}(f,\|H^{k}_{2}\|)&\subseteq\mu_{\ast}\mathcal{O}_{\tilde{X}}(-E+K_{\tilde{X}/X}-\lfloor\frac{1}{p}G)\rfloor)\\
&\subseteq\mu_{\ast}\mathcal{O}_{\tilde{X}}(K_{\tilde{X}/X}-\lfloor\frac{1}{p}H)\rfloor)\\
&=\mathscr{I}(f,\|H^{m}_{1}\otimes H^{k}_{2}\|).
\end{split}
\]

(iii) Fix $p\gg0$ and divisible enough that computes both of the multiplier ideals $\mathscr{I}(f,k\|H_{1}\|)$ and $\mathscr{I}(f,\|H^{k}_{1}\|)$. Then
\[
\begin{split}
\mathscr{I}(f,\|H^{k}_{1}\|)&=\mathscr{I}(f,\frac{1}{p}|H^{pk}_{1}|)\\
&=\mathscr{I}(f,\frac{k}{pk}|H^{pk}_{1}|)\\
&=\mathscr{I}(f,k\|H_{1}\|).
\end{split}
\]
Now we have
\[
\begin{split}
\mathscr{I}(f,\|H^{k}_{1}\|)&=\mathscr{I}(f,k\|H_{1}\|)\\
&\supseteq\mathscr{I}(f,(k+1)\|H_{1}\|)\\
&=\mathscr{I}(f,\|H^{k+1}_{1}\|).
\end{split}
\]

(iv) Fix $p\gg0$ and divisible enough that computes both of the multiplier ideals $\mathscr{I}(f,\|L\|)$ and $\mathscr{I}(\|L\|)$. Let $\{u_{i}\}$ be a basis of $H^{0}(X,L^{p})$. Then for a general fibre $F$ of $f$, $u_{i}|_{F}$ is a section of $H^{0}(F,L^{p})$. Hence $\{u_{i}\}$ forms an ideal that is contained in the relative base-ideal of $|L^{p}|$. Now the inclusion is obvious.
\end{proof}
\end{proposition}

As a by-product of the formula
\[
\mathscr{I}(f,\|H^{k}_{1}\|)=\mathscr{I}(f,k\|H_{1}\|)
\]
in (iii), if $\varphi=\{\varphi_{U}\}$ is the collection of metrics associated to $\mathscr{I}(f,\|H_{1}\|)$, $k\varphi=\{k\varphi_{U}\}$ is the collection of metrics associated to $\mathscr{I}(f,\|H^{k}_{1}\|)$. 

\subsection{Fibration}
\label{sec:fibration}
Next, we recall the definition of a fibre product \cite{Har77}.

\begin{definition}\label{d25}
Let $f:X\rightarrow Y$ be a fibration between two projective manifolds $X$ and $Y$. The fibre product, denoted by $(X\times_{Y}X,p^{2}_{1},p^{2}_{2})$, is a projective manifold coupled with two morphisms (we will also call the manifold $X\times_{Y}X$ itself the fibre product if nothing is confused) that satisfies the following properties:

1.The diagram
\[
\begin{array}[c]{ccc}
X\times_{Y}X&\stackrel{p^{2}_{2}}{\rightarrow}&X\\
\scriptstyle{p^{2}_{1}}\downarrow&&\downarrow\scriptstyle{f}\\
X&\stackrel{f}{\rightarrow}&Y
\end{array}
\]
commutes.

2.If there is another projective manifold $Z$ with morphisms $q_{1}, q_{2}$ such that the diagram
\[
\begin{array}[c]{ccc}
Z&\stackrel{q_{2}}{\rightarrow}&X\\
\scriptstyle{q_{1}}\downarrow&&\downarrow\scriptstyle{f}\\
X&\stackrel{f}{\rightarrow}&Y
\end{array}
\]
commutes, then there must exist a unique $g:Z\rightarrow X\times_{Y}X$ such that $p^{2}_{1}\circ g=q_{1},p^{2}_{2}\circ g=q_{2}$.

We inductively define the $m$-fold fibre product, and denote it by $\underbrace{X\times_{Y}\cdot\cdot\cdot\times_{Y}X}_{m}$. Then, the two projections are denoted by
\[
 p^{m}_{1}:X\times_{Y}\cdot\cdot\cdot\times_{Y}X\rightarrow X
\]
and
\[
 p^{m}_{2}:\underbrace{X\times_{Y}\cdot\cdot\cdot\times_{Y}X}_{m}\rightarrow \underbrace{X\times_{Y}\cdot\cdot\cdot\times_{Y}X}_{m-1}
\]
respectively.
\end{definition}

The meaning of the fibre product is clear from the viewpoint of geometry. In particular, if $y$ is a regular value of $f$,
\[
 (X\times_{Y}\cdot\cdot\cdot\times_{Y}X)_{y}=X_{y}\times\cdot\cdot\cdot\times X_{y}.
\]

We collect the following two lemmas from \cite{Har77} without proof for the later use.

\begin{lemma}[(Projection formula)]\label{l21}
If $f:X\rightarrow Y$ is a holomorphic morphism between two projective manifolds $X$ and $Y$, $\mathcal{F}$ is a coherent sheaf on $X$, and $\mathcal{E}$ is a locally free sheaf on $Y$, then there is a natural isomorphism
\[
  f_{\ast}(\mathcal{F}\otimes f^{\ast}\mathcal{E})\cong f_{\ast}(\mathcal{F})\otimes\mathcal{E}.
\]
\end{lemma}

\begin{lemma}[(Base change)]\label{l22}
Assume that $f:X\rightarrow Y, v:X^{\prime}\rightarrow X$ and $g:X^{\prime}\rightarrow Y^{\prime}$ are holomorphic morphisms between projective manifolds $X,X^{\prime},Y$ and $Y^{\prime}$. Let $\mathcal{F}$ be a coherent sheaf on $X$. $u:Y^{\prime}\rightarrow Y$ is a smooth morphism, such that
\[
\begin{array}[c]{ccc}
X^{\prime}&\stackrel{v}{\rightarrow}&X\\
\scriptstyle{g}\downarrow&&\downarrow\scriptstyle{f}\\
Y^{\prime}&\stackrel{u}{\rightarrow}&Y
\end{array}
\]
commutes. Then for all $q\geqslant0$ there is a natural isomorphism
\[
   u^{\ast}R^{q}f_{\ast}(\mathcal{F})\cong R^{q}g_{\ast}(v^{\ast}\mathcal{F}).
\]
\end{lemma}

\section{The injectivity theorem and the vanishing theorem}
\label{sec:injectivity}

\subsection{The injectivity theorem}
Theorem \ref{t13} is a direct consequence of the following variant of the Koll\'{a}r-type injectivity theorem developed in \cite{Eno93,Fuj12,GoM17,Ko86a,Ko86b,Mat14,Mat15a,Mat18}.

\begin{theorem}\label{t31}
Let $(H,\varphi_{H})$ and $(M,\varphi_{M})$ be line bundles with
(singular) metrics on a projective manifold $X$. Assume the following conditions:
\begin{enumerate}
\item[(1)] $i\Theta_{H,\varphi_{H}}\geqslant0$ and $i\Theta_{M,\varphi_{M}}\geqslant\gamma$ for some smooth real $(1,1)$-form $\gamma$ on $X$; and
\item[(2)] $i\Theta_{H,\varphi_{H}}\geqslant\varepsilon i\Theta_{M,\varphi_{M}}$ for some positive number $\varepsilon$.
\end{enumerate}

Then for a (non-zero) section $s$ of $M$ with $\sup_{X}|s|^{2}e^{-\varphi_{M}}<\infty$, the multiplication map induced by the tensor product with $s$
\[
\Phi:H^{q}(X,K_{X}\otimes H\otimes\mathscr{I}(\varphi_{H}))\rightarrow H^{q}(X,K_{X}\otimes H\otimes M\otimes\mathscr{I}(\varphi_{H}+\varphi_{M}))
\]
is well-defined and injective for any $q\geqslant0$. 

\begin{proof}
The proof is nothing but repeating the argument in, say \cite{GoM17}, so we omit it here.
\end{proof}
\end{theorem}

Now we prove Theorem \ref{t13}.
\begin{proof}[Proof of Theorem \ref{t13}]
We claim that for a general fibre $F$ of $f$, if $u$ is a section of 
\[
H^{0}(F,L^{p}), 
\]
it extends to a global section $\tilde{u}$ of $H^{0}(X,L^{p}\otimes A^{\prime})$. 

In order to prove this claim, we first recall the Ohsawa--Takegoshi extension theorem as follows:
\begin{theorem}[(Theorem 1, \cite{Man93})]\label{t32}
Let $X$ be a projective manifold, and let $Z\subset X$ be the zero set of a holomorphic section $s\in H^{0}(X,E)$ of a vector bundle $E\rightarrow X$; the subset $Z$ is assumed to be non-singular and of codimension $r=\textrm{rank }(E)$. Let $(G,\varphi)$ be a line bundle on $X$, endowed with a (singular) metric $\varphi$, such that
\begin{enumerate}
\item[(a)] $i\Theta_{G,\varphi}+i\partial\bar{\partial}\log|s|^{2}\geqslant0$ on $X$;
\item[(b)] $i\Theta_{G,\varphi}+i\partial\bar{\partial}\log|s|^{2}\geqslant\frac{1}{\alpha}\frac{<i\Theta_{E}s,s>}{|s|^{2}}$ for some $\alpha\geqslant1$; and 
\item[(c)] $|s|^{2}\leqslant e^{-\alpha}$ on $X$, and the restriction of $\varphi$ on $Z$ is well-defined.
\end{enumerate}

Then every section $u\in H^{0}(Z,K_{Z}\otimes G\otimes\mathscr{I}(\varphi|_{Z}))$ admits an extension $\tilde{u}$ to $X$ such that 
\[
\int_{X}\frac{\tilde{u}\wedge\bar{\tilde{u}}e^{-\varphi}}{|s|^{2r}(-\log|s|)^{2}}\leqslant C_{\alpha}\int_{Z}\frac{|u|^{2}e^{-\varphi}}{|\Lambda^{r}(ds)|^{2}},
\]
provided the right hand side is finite.
\end{theorem}

Now the bundle that we are interested in can be decomposed as
\[
L^{p}\otimes A^{\prime}=K_{X}\otimes -K_{X/Y}\otimes L^{p}\otimes f^{\ast}(A\otimes -K_{Y}).
\]
(The bundle $A'$ is chosen in a moment.) Our goal is to show that it is effective by extending the section $0\neq u\in H^{0}(F,L^{p}\otimes\mathscr{I}(\tau))$. (Remember that $\mathscr{I}(\tau)=\mathcal{O}_{X}$.) We choose now the bundle $A$ positive enough so that
\begin{enumerate}
\item[(1)] $H^{0}(Y,A)\neq0$; and
\item[(2)] the point $y\in Y$ such that $f^{-1}(y)=F$ is the common zero set of the sections $s=(s_{j})$ of an ample line bundle $B\rightarrow Y$ and $A\otimes -K_{Y}\geqslant B^{2}$, in the sense that the difference is an ample line bundle.
\end{enumerate}
Obviously $A$ as well as $A'=f^{\ast}A$ is independent of $f$ and $L$. Property (2) gives a smooth metric $\varphi_{A\otimes -K_{Y}}$ with positive curvature. The bundle $-K_{X/Y}\otimes L^{p}\otimes f^{\ast}(A\otimes -K_{Y})$ is endowed with the metric $\tau\otimes f^{\ast}\varphi_{A\otimes -K_{Y}}$; its curvature is semipositive on $X$, and the restriction to $F$ is well-defined. The section we want to extend is $v:=u\otimes s_{A^{\prime}}$, where $s_{A^{\prime}}$ is the pullback of some nonzero section given by property (1) above. By property (2), the positivity conditions in Theorem \ref{t32} are satisfied with the bundle $G$ given by
\[
G=-K_{X/Y}\otimes L^{p}\otimes f^{\ast}(A\otimes -K_{Y}).
\]
Since
\[
i\Theta_{G,\tau\otimes f^{\ast}\varphi_{A\otimes -K_{Y}}}+i\partial\bar{\partial}\log|s|^{2}\geqslant f^{\ast}\Theta_{A\otimes -K_{Y}\otimes B^{-1}}.
\]
The right-hand side is semipositive, and it dominates the bundle $B$; thus the requirements (a)-(c) are verified.

Now the integrability condition is obviously acceptable, and by Theorem \ref{t32} we can extend the section $v$ over $X$. The claim is proved. 

We return to Theorem \ref{t13}. Note $p$ computes $\mathscr{I}(f,\|L\|)$. Remember the discussion in Sect.\ref{sec:asymptotic} and keep the notations there, there exists a collection of singular metrics $\varphi=\{\varphi_{U}\}$ defined by the sections of $\Gamma(f^{-1}(U),L^{p})$, say $\{u_{i,U}\}$, with $i\Theta_{L,\varphi_{U}}\geqslant0$ and 
\[
\mathscr{I}(\varphi)=\mathscr{I}(f,\|L\|).
\] 
Let $\{s_{j}\}$ be the sections of $A^{\prime}$ that generates $A^{\prime}$. Due to the claim before, all of the sections $\{u_{i,U}\otimes s_{j}\}$ extend over $X$ as the sections $\{v_{ij}\}$ of $H^{0}(X,L^{p}\otimes A^{\prime})$. Certainly they together define a (singular) metric $\chi$ on $L^{p}\otimes A^{\prime}$ with positive curvature current. In particular, since $\mathcal{I}(\{u_{i,U}\})=\mathcal{I}(\{u_{i,U}\otimes s_{j}\})=\mathcal{I}(\{v_{ij}\})$ on $f^{-1}(U)$,
\[
\mathscr{I}(\frac{1}{p}\chi)=\mathscr{I}(\varphi)=\mathscr{I}(f,\|L\|)\textrm{ on }f^{-1}(U)\textrm{ hence everywhere}.
\] 

Now let $\psi$ be a smooth metric on $A^{\prime}$ with semipositive curvature. Let
\[
 (H,\varphi_{H})=(L\otimes A^{\prime},\frac{1}{p}\chi+\frac{p-1}{p}\psi)\textrm{ and }(M,\varphi_{M})=(A^{\prime},\psi).
\] 
The requirements (1) and (2) of Theorem \ref{t31} are easy to verified. We then obtain the desired injectivity result by Theorem \ref{t31}.
\end{proof}

\subsection{The vanishing theorem}
Now we are ready to prove Theorem \ref{t14}.
\begin{proof}[Proof of Theorem \ref{t14}]
Let $\varphi=\{\varphi_{U}\}$ be the collection of metrics on $L|_{f^{-1}(U)}$ such that 
\[
\mathscr{I}(\varphi)=\mathscr{I}(f,\|L\|).
\] 
Let $A$ be the ample line bundle on $Y$ (independent of $f$ and $L$) picked in Theorem \ref{t13}.

By asymptotic Serre vanishing theorem \cite{Har77}, we can choose a positive integer $m_{0}$ such that for all $m\geqslant m_{0}$,
\[
    H^{i}(Y,R^{q}f_{\ast}(K_{X}\otimes L\otimes\mathscr{I}(\varphi))\otimes A^{m})=0
\]
for $i>0,\, q\geqslant0$. Fix an integer $m$ such that $m\geqslant m_{0}$ and  $A^{m}$ is very ample.

We prove the theorem by induction on $n=\dim Y$, the case $n=0$ being trivial.
Denote $A'=f^\ast (A)$ and let $H^{\prime}\in|(A^{\prime})^{m}|$ be the pull back of a general divisor $H\in|A^{m}|$. It follows from Bertini's theorem \cite{Har77} that we can assume $H$ is integral and $H^{\prime}$ is smooth (though possibly disconnected). Moreover, we can arrange the things that $\mathscr{I}(\varphi|_{H'})=\mathscr{I}(\varphi)|_{H'}$ by \cite{FuM16}, Theorem 1.10. Then we have a short exact sequence
\begin{equation}\label{e31}
\begin{split}
0&\rightarrow K_{X}\otimes L\otimes A^{\prime}\otimes\mathscr{I}(\varphi)\rightarrow K_{X}\otimes L\otimes (A^{\prime})^{m+1}\otimes\mathscr{I}(\varphi)\\
&\rightarrow K_{H^{\prime}}\otimes L\otimes A^{\prime}|_{H^{\prime}}\otimes\mathscr{I}(\varphi|_{H'})\rightarrow0
\end{split}
\end{equation}
which is induced by multiplication with a section defining $H^{\prime}$. We get from this short exact sequence a long exact sequence
\begin{equation}\label{32}
\begin{split}
0&\rightarrow f_{\ast}(K_{X}\otimes L\otimes A^{\prime}\otimes\mathscr{I}(\varphi))\rightarrow f_{\ast}(K_{X}\otimes L\otimes (A^{\prime})^{m+1}\otimes\mathscr{I}(\varphi))\\
&\rightarrow f_{\ast}(K_{H^{\prime}}\otimes L\otimes A^{\prime}|_{H^{\prime}}\otimes\mathscr{I}(\varphi|_{H'}))\rightarrow R^{1}f_{\ast}(K_{X}\otimes L\otimes A^{\prime}\otimes\mathscr{I}(\varphi))\\
&\rightarrow R^{1}f_{\ast}(K_{X}\otimes L\otimes (A^{\prime})^{m+1}\otimes\mathscr{I}(\varphi))\rightarrow\cdots
\end{split}
\end{equation}
By \cite{Wu21}, Theorem 6.1 all the higher direct images of $K_{X}\otimes L\otimes A^{\prime}\otimes\mathscr{I}(\varphi)$ are torsion-free. Also clearly the sheaves $R^{q}f_{\ast}(K_{H^{\prime}}\otimes L\otimes A^{\prime}|_{H^{\prime}}\otimes\mathscr{I}(\varphi|_{H'}))$ are torsion on $H$. Hence the long exact sequence (\ref{32}) can be split into a family of short exact sequences: for all $q\geq 0$,
\begin{equation}\label{e33}
\begin{split}
0&\rightarrow R^{q}f_{\ast}(K_{X}\otimes L\otimes A^{\prime}\otimes\mathscr{I}(\varphi))\rightarrow R^{q}f_{\ast}(K_{X}\otimes L\otimes (A^{\prime})^{m+1}\otimes\mathscr{I}(\varphi))\\
&\rightarrow R^{q}f_{\ast}(K_{H^{\prime}}\otimes L\otimes A^{\prime}|_{H^{\prime}}\otimes\mathscr{I}(\varphi|_{H'}))\rightarrow0.
\end{split}
\end{equation}
On the other hand, applying the inductive hypothesis to each connected component of $H^{\prime}$,
we have that for all $i\geqslant1$
\[
   H^{i}(Y,R^{q}f_{\ast}(K_{H^{\prime}}\otimes L\otimes A^{\prime}|_{H^{\prime}}\otimes\mathscr{I}(\varphi|_{H'})))=0.
\]
Furthermore, by the choice of $m$ we also have for all $i\geqslant1$
\begin{equation}\label{34}
    H^{i}(Y,R^{q}f_{\ast}(K_{X}\otimes L\otimes (A^{\prime})^{m+1}\otimes\mathscr{I}(\varphi)))=0.
\end{equation}
Now by taking the cohomology long exact sequence from the short exact sequence (\ref{e33}), we find for every $i>1$
\[
   H^{i}(Y,R^{q}f_{\ast}(K_{X}\otimes L\otimes A^{\prime}\otimes\mathscr{I}(\varphi)))=0.
\]
This proves the theorem for the cases where $i>1$.

To prove the case where $i=1$, we denote
\[
   B_{l}:=H^{1}(Y,R^{q}f_{\ast}(K_{X}\otimes L\otimes (A^{\prime})^{l}\otimes\mathscr{I}(\varphi))).
\]
By identity (\ref{34}) for $i=1$, we have $B_{m+1}=0$. Hence we consider the following commutative diagram:
\begin{equation*}
\xymatrix{
    B_{1}\ar[d] & \stackrel{\phi}{\longrightarrow} & H^{q+1}(X,K_{X}\otimes L \otimes A^{\prime}\otimes\mathscr{I}(\varphi))\ar[d]^{\psi} \\
    B_{m+1} & {\longrightarrow} &H^{q+1}(X,K_{X}\otimes L\otimes(A^{\prime})^{m+1}\otimes\mathscr{I}(\varphi))
    }
\end{equation*}
Here the horizontal maps are the canonical injective maps coming out of the Leray spectral sequence \cite{Har77}, and the vertical maps are induced by multiplication with sections defining $H^{\prime}$ and $H$ respectively. By Theorem \ref{t13} the map $\psi$ is injective, and hence the composition $\psi\circ\phi$ is also injective. So $B_1=0$ and we finish the proof of the theorem for the case where $i=1$.
\end{proof}

Using Theorem \ref{t14}, we can prove the global generation of the higher direct images. We first review the definition and a basic result of the Castelnuovo--Mumford regularity \cite{Mum66}.

\begin{definition}\label{d31}
Let $X$ be a projective manifold and $L$ an ample and globally generated line bundle on $X$. Given an integer $m$, a coherent sheaf $F$ on $X$ is $m$-regular with respect to $L$ if for all $i\geqslant1$
\[
H^{i}(X,F\otimes L^{m-i})=0.
\]
\end{definition}

\begin{theorem}(Mumford, \cite{Mum66})\label{t33}
Let $X$ be a projective manifold and $L$ an ample and globally generated line bundle on $X$. If $F$ is a coherent sheaf on $X$ that is $m$-regular with respect to $L$, then the sheaf $F\otimes L^{m}$ is globally generated.
\end{theorem}

After this, we can prove Corollary \ref{c11}.
\begin{proof}[Proof of Corollary \ref{c11}]
It follows from Theorem \ref{t14} that for every $i\geqslant1$
\[
  H^{i}(Y,R^{q}f_{\ast}(K_{X}\otimes L\otimes\mathscr{I}(f,\|L\|))\otimes A^{m-i}\otimes A^{\prime})=0.
\]
Hence the sheaf $R^{q}f_{\ast}(K_{X}\otimes L\otimes\mathscr{I}(f,\|L\|))\otimes A^{m}\otimes A^{\prime}$ is $0$-regular with respect to $A$. So it is globally generated by Theorem \ref{t33}.
\end{proof}

\section{The positivity of direct images}
\label{sec:direct}
In this section, we shall prove Theorem \ref{t12} following Veihweg's strategy in \cite{Vie82b,Vie83}. Note $f$ is furthermore supposed to be smooth here. Recall that $\varphi=\{\varphi_{U}\}$ is the collection of metrics on $L|_{f^{-1}(U)}$ such that $\mathscr{I}(\varphi)=\mathscr{I}(f,\|L\|)$. The following observation is needed:
\begin{lemma}\label{l41}
For any positive integer $m$, consider the $m$-fold fibre product 
\[
f_{m}:X_{m}=X\times_{Y}\cdot\cdot\cdot\times_{Y}X\rightarrow Y.
\] 
Using the same notation as Definition \ref{d25}, we have
\begin{enumerate}
\item $(f_{m})_{\ast}(K_{X_{m}/Y}\otimes(p^{m}_{1}\otimes p^{m-1}_{1}p^{m}_{2}\otimes\cdot\cdot\cdot\otimes p^{1}_{2}p^{2}_{2}...p^{m}_{2})^{\ast}(L))=f_{\ast}(K_{X/Y}\otimes L)^{\otimes m}$.
\item $\varphi_{m}=(p^{m}_{1}+p^{m-1}_{1}p^{m}_{2}+\cdot\cdot\cdot+p^{1}_{2}p^{2}_{2}...p^{m}_{2})^{\ast}\varphi$ is a collection of metrics on
\[
    L_{m}:=(p^{m}_{1}\otimes p^{m-1}_{1}p^{m}_{2}\otimes\cdot\cdot\cdot\otimes p^{1}_{2}p^{2}_{2}...p^{m}_{2})^{\ast}(L)
\]
such that $\mathscr{I}(\varphi_{m})=\mathscr{I}(f_{m},\|L_{m}\|)$. In particular, if $p$ is an integer that computes $\mathscr{I}(f,\|L\|)$, it also computes $\mathscr{I}(f_{m},\|L_{m}\|)$.
\item(Subadditivity) $\mathscr{I}(\varphi_{m})\subset(p^{m}_{1}\otimes p^{m-1}_{1}p^{m}_{2}\otimes\cdot\cdot\cdot\otimes p^{1}_{2}p^{2}_{2}...p^{m}_{2})^{\ast}\mathscr{I}(\varphi)$.
\end{enumerate}
\begin{proof}

(i) We simply prove it with $m=3$, and the general case follows in the same way. The calculation involves nothing but the repeated use of Lemmas \ref{l21} and \ref{l22}. We also need the following two facts in \cite{Har77}:

1. $(g\circ f)_{\ast}=g_{\ast}f_{\ast}$ for arbitrary morphisms $f$ and $g$;

2. $K_{X_{m}/Y}=(p^{m}_{1})^{\ast}K_{X/Y}\otimes (p^{m}_{2})^{\ast}K_{X_{m-1}/Y}$.

Then, we complete the proof by carefully chasing the diagram.
\[
\begin{split}
 &(f_{3})_{\ast}(K_{X_{3}/Y}\otimes (p^{3}_{1})^{\ast}L\otimes (p^{3}_{2})^{\ast}(p^{2}_{1})^{\ast}L\otimes (p^{3}_{2})^{\ast}(p^{2}_{2})^{\ast}L)\\
=&f_{\ast}(p^{3}_{1})_{\ast}((p^{3}_{1})^{\ast}K_{X/Y}\otimes (p^{3}_{2})^{\ast}K_{X_{2}/Y}\otimes (p^{3}_{1})^{\ast}L\otimes (p^{3}_{2})^{\ast}(p^{2}_{1})^{\ast}L\otimes (p^{3}_{2})^{\ast}(p^{2}_{2})^{\ast}L)\\
=&f_{\ast}(K_{X/Y}\otimes L\otimes (p^{3}_{1})_{\ast}((p^{3}_{2})^{\ast}K_{X_{2}/Y}\otimes (p^{3}_{2})^{\ast}(p^{2}_{1})^{\ast}L\otimes (p^{3}_{2})^{\ast}(p^{2}_{2})^{\ast}L))\\
=&f_{\ast}(K_{X/Y}\otimes L\otimes f^{\ast}(f_{2})_{\ast}(K_{X_{2}/Y}\otimes (p^{2}_{1})^{\ast}L\otimes (p^{2}_{2})^{\ast}L)\\
=&f_{\ast}(K_{X/Y}\otimes L\otimes f^{\ast}(f_{\ast}(p^{2}_{1})_{\ast})((p^{2}_{1})^{\ast}K_{X/Y}\otimes (p^{2}_{2})^{\ast}K_{X/Y}\otimes (p^{2}_{1})^{\ast}L\otimes (p^{2}_{2})^{\ast}L)\\
=&f_{\ast}(K_{X/Y}\otimes L\otimes f^{\ast}f_{\ast}(K_{X/Y}\otimes L\otimes (p^{2}_{1})_{\ast}((p^{2}_{2})^{\ast}K_{X/Y}\otimes (p^{2}_{2})^{\ast}L))\\
=&f_{\ast}(K_{X/Y}\otimes L)\otimes f_{\ast}(K_{X/Y}\otimes L\otimes f^{\ast}f_{\ast}(K_{X/Y}\otimes L))\\
=&f_{\ast}(K_{X/Y}\otimes L)^{\otimes3}.
\end{split}
\]

(ii) Similar computation with (i) also implies that $(f_{m})_{\ast}(L_{m})=f_{\ast}(L)^{\otimes m}$. Let $U$ be a local coordinate ball of $Y$. Then any section $u$ of $\Gamma(f^{-1}_{m}(U),L_{m})$ decomposes as
\[
u=u_{1}\otimes u_{2}\otimes \cdot\cdot\cdot\otimes u_{m}, 
\] 
where $u_{1},...,u_{m}$ are sections of $\Gamma(f^{-1}(U),L)$. By definition, we have
\[
\mathscr{I}(\varphi_{m})=\mathscr{I}(f_{m},\|L_{m}\|).
\]

(iii) It comes from the subadditivity of the multiplier ideal sheaves proved in \cite{DEL00}. In fact, since
\[
    \varphi_{m}=(p^{m}_{1}+p^{m-1}_{1}p^{m}_{2}+\cdot\cdot\cdot+p^{1}_{2}p^{2}_{2}...p^{m}_{2})^{\ast}\varphi,
\]
we have
\begin{equation}\label{e41}
    \mathscr{I}(\varphi_{m})\subset\mathscr{I}((p^{m}_{1})^{\ast}\varphi\cdot\mathscr{I}((p^{m-1}_{1}p^{m}_{2})^{\ast}\varphi)\cdot\cdots\cdot\mathscr{I}((p^{1}_{2}p^{2}_{2}...p^{m}_{2})^{\ast}\varphi)
\end{equation}
by the main result (Theorem 2.6) in \cite{DEL00}. One more application of Theorem 2.6 in \cite{DEL00} implies that
\begin{equation}\label{e42}
\begin{split}
\mathscr{I}((p^{m}_{1})^{\ast}\varphi)&=(p^{m}_{1})^{\ast}\mathscr{I}(\varphi),\\
\mathscr{I}((p^{m-1}_{1}p^{m}_{2})^{\ast}\varphi)&=(p^{m-1}_{1}p^{m}_{2})^{\ast}\mathscr{I}(\varphi),\\
\mathscr{I}((p^{m-2}_{1}p^{m-1}_{2}p^{m}_{2})^{\ast}\varphi)&=(p^{m-2}_{1}p^{m-1}_{2}p^{m}_{2})^{\ast}\mathscr{I}(\varphi),\\
    &\cdots\\
    \mathscr{I}((p^{1}_{2}p^{2}_{2}...p^{m}_{2})^{\ast}\varphi)&=(p^{1}_{2}p^{2}_{2}...p^{m}_{2})^{\ast}\mathscr{I}(\varphi).
\end{split}
\end{equation}

Indeed, let $y\in Y$. Take a local coordinate neighbourhood $U$ of $y$, we have
\[
    X|_{f^{-1}(U)}=U\times X_{y}
\]
and
\[
    X_{m}|_{f^{-1}_{m}(U)}=\underbrace{X_{y}\times\cdots\times X_{y}}_{m-1}\times(U\times X_{y}).
\]
Therefore, locally $X_{m}$ can be regarded as the product of two manifolds:
\[
X_{1}=\underbrace{X_{y}\times\cdots\times X_{y}}_{m-1} \textrm{ and } X_{2}=U\times X_{y}.
\]
Let $\phi_{1}=1$, which is a function on $X_{1}$. Let $\phi_{2}=\varphi_{U}$, which is a function on $X_{2}$. Apply the first statement of Theorem 2.6 in \cite{DEL00} to $(X_{1},\phi_{1})$ and $(X_{2},\phi_{2})$, we obtain
\[
\mathscr{I}((p^{m}_{1})^{\ast}\varphi)=(p^{m}_{1})^{\ast}\mathscr{I}(\varphi)
\]
at $y$ hence everywhere.  

The other formulas are the same. Combined with (\ref{e41}) and (\ref{e42}), the proof is complete.
\end{proof}
\end{lemma}

Before introducing the next lemma, we need to fix some notations. We denote the $m$-fold fibre product $\underbrace{X\times_{Y}\cdot\cdot\cdot\times_{Y}X}_{m}$ by
\[
f_{m}:X_{m}\rightarrow Y.
\]
Furthermore, let
\[
L_{m}:=(p^{m}_{1}\otimes p^{m-1}_{1}p^{m}_{2}\otimes\cdot\cdot\cdot\otimes p^{1}_{2}p^{2}_{2}...p^{m}_{2})^{\ast}(L),
\]
and $\varphi_{m}=\{\varphi_{U,m}\}$ be the collection of metrics induced by $\varphi=\{\varphi_{U}\}$. Now, given $m$ sections $u_{1},...,u_{m}$ of
\[
    \Gamma(U,f_{\ast}(K_{X/Y}\otimes L)),
\]
by Lemma \ref{l41} they together induce a section
\[
u^{\otimes m}:=(p^{m}_{1})^{\ast}u_{1}\otimes(p^{m-1}_{1}p^{m}_{2})^{\ast}u_{2}\otimes \cdot\cdot\cdot\otimes(p^{1}_{2}p^{2}_{2}...p^{m}_{2})^{\ast}u_{m}
\]
of
\[
\Gamma(U,(f_{m})_{\ast}(K_{X_{m}/Y}\otimes L_{m}))=\Gamma(f^{-1}_{m}(U),K_{X_{m}/Y}\otimes L_{m}).
\]
The next lemma shows that we even have
\[
u^{\otimes m}\in\Gamma(U,(f_{m})_{\ast}(K_{X_{m}/Y}\otimes L_{m}\otimes\mathscr{I}(\varphi_{m}))).
\]

\begin{lemma}\label{l42}
Keep the notations. Let $\varphi=\{\varphi_{U}\}$ and $\psi=\{\psi_{U}\}$ be the collection of metrics associated to $\mathscr{I}(f,\|L\|)$ and $\mathscr{I}(f,\|K_{X/Y}\otimes L\|)$ respectively. Assume that $\varphi$ is less singular than $\psi$, i.e.
\[
\varphi_{U}\preceq\psi_{U}\textrm{ for every }U.
\] 
Then
\[
    \int_{f^{-1}_{m}(U)}|u^{\otimes m}|^{2}e^{-\varphi_{U,m}}
\]
is finite. In other words,
\[
    u^{\otimes m}\in\Gamma(U,(f_{m})_{\ast}(K_{X_{m}/Y}\otimes L_{m}\otimes\mathscr{I}(\varphi_{m}))).
\]
\begin{proof}
Let $y\in Y$ be an arbitrary point. Take a local coordinate neighbourhood $U$ of $y$, such that
\[
X|_{f^{-1}(U)}=U\times X_{y}.
\] 
If we take the local coordinate on $f^{-1}(U)$ to be
\[
((y_{1},...,y_{n}),(x_{1},...,x_{l})),
\]
$\varphi_{U}$ can be written on $f^{-1}(U)$ as:
\[
   \varphi_{U}=\varphi_{U}((y_{1},...,y_{n}),(x_{1},...,x_{l})).
\]
Moreover, we have 
\[
f^{-1}_{m}(U)=U\times X^{1}_{y}\times\cdot\cdot\cdot\times X^{m}_{y}.
\] 
Here, we add the superscript $\{1,...,m\}$ to differentiate the fibres. Then, the corresponding local coordinate on $f^{-1}_{m}(U)$ will be
\[
((y_{1},...,y_{n}),(x^{1}_{1},...,x^{1}_{l}),...,(x^{m}_{1},...,x^{m}_{l})),
\]
and $\varphi_{U,m}$ becomes
\[
\begin{split}
\varphi_{U,m}&=(p^{m}_{1}+p^{m-1}_{1}p^{m}_{2}+p^{m-2}_{1}p^{m-1}_{2}p^{m}_{2}+...+ p^{1}_{2}p^{2}_{2}...p^{m}_{2})^{\ast}\varphi_{U}\\
           &=\sum_{j}\varphi_{U}((y_{1},...,y_{n}),(x^{j}_{1},...,x^{j}_{l})).
\end{split}
\]
We claim that, for any section $u^{\otimes m}$ of
\[
   \Gamma(U,(f_{m})_{\ast}(K_{X_{m}/Y}\otimes L_{m}))=\Gamma(f^{-1}_{m}(U),K_{X_{m}/Y}\otimes L_{m})
\]
defined above, the integral
\[
\begin{split}
\int_{U}\int_{X^{1}_{y}\times...\times X^{m}_{y}}|u^{\otimes m}|^{2}e^{-\varphi_{U,m}}
\end{split}
\]
is finite. In fact, we have
\[
\begin{split}
&\int_{U}\int_{X^{1}_{y}\times...\times X^{m}_{y}}|u^{\otimes m}|^{2}e^{-\varphi_{U,m}}\\
=&\int_{U}\int_{X^{1}_{y}\times...\times X^{m}_{y}}|u^{\otimes m}|^{2}e^{-\sum_{j}\varphi_{U}((y_{1},...,y_{n}),(x^{j}_{1},...,x^{j}_{l}))}\\
=&\int_{U}\prod^{m}_{j}\int_{X^{j}_{y}}|u_{j}|^{2}e^{-\varphi_{U}((y_{1},...,y_{n}),(x^{j}_{1},...,x^{j}_{l}))}\\
\leqslant&\prod^{m}_{j}(\int_{U}(\int_{X^{j}_{y}}|u_{j}|^{2}e^{-\varphi_{U}((y_{1},...,y_{n}),(x^{j}_{1},...,x^{j}_{l}))})^{m})^{1/m}.
\end{split}
\]
This last inequality is due to H\"{o}lder's inequality. Since
\[
    \int_{U}(\int_{X^{j}_{y}}|u_{j}|^{2}e^{-\varphi_{U}((y_{1},...,y_{n}),(x^{j}_{1},...,x^{j}_{l}))})^{m}=\int_{U}(\int_{X_{y}}|u_{j}|^{2}e^{-\varphi_{U}((y_{1},...,y_{n}),(x_{1},...,x_{l}))})^{m}
\]
for every $j$, this integral is finite by the definition of $\varphi_{U}$ (see Sect.\ref{sec:asymptotic}). 

In fact, since $u_{j}$ is a section of $\Gamma(f^{-1}(U),K_{X/Y}\otimes L)$, $u_{j}\in\mathfrak{a}_{f}\subseteq \mathscr{I}(f,\|K_{X/Y}\otimes L\|)$ by Proposition \ref{p22}, (i). Here $\mathfrak{a}_{f}$ is the base-ideal of $|K_{X/Y}\otimes L|$ relative to $f$. As a result, $|u_{j}|^{2}e^{-\psi_{U}((y_{1},...,y_{n}),(x_{1},...,x_{l}))}$ will be bounded. On the other hand, $\varphi_{U}$ is less singular than $\psi_{U}$ by assumption. So $|u_{j}|^{2}e^{-\varphi_{U}((y_{1},...,y_{n}),(x_{1},...,x_{l}))}$ is also bounded on $f^{-1}(U)$. We have finished the proof of this claim.

Then, we conclude that
\[
    \int_{f^{-1}_{m}(U)}|u^{\otimes m}|^{2}e^{-\varphi_{U,m}}
\]
is also finite. Indeed, let $Z:=\{y\in Y; \varphi_{U}|_{X_{y}}\equiv-\infty\}$, which is a set of measure zero. We have
\begin{equation}\label{e43}
\begin{split}
    \int_{f^{-1}_{m}(U)\setminus f^{-1}_{m}(Z)}|u^{\otimes m}|^{2}e^{-\varphi_{U,m}}&=\int_{(U\setminus Z)\times X^{1}_{y}\times\cdot\cdot\cdot\times X^{m}_{y}}|u^{\otimes m}|^{2}e^{-\varphi_{U,m}}\\
    &\leqslant\int_{U}\int_{X^{1}_{y}\times\cdot\cdot\cdot\times X^{m}_{y}}|u^{\otimes m}|^{2}e^{-\varphi_{U,m}}.
\end{split}
\end{equation}
Since the right hand side of the inequality is finite, the inequality is actually an equality. Therefore we conclude that
\[
    \int_{f^{-1}_{m}(U)}|u^{\otimes m}|^{2}e^{-\varphi_{U,m}}:=\int_{f^{-1}_{m}(U)\setminus f^{-1}_{m}(Z)}|u^{\otimes m}|^{2}e^{-\varphi_{U,m}}<+\infty.
\]
\end{proof}
\end{lemma}

Now we turn to Theorem \ref{t12}.

\begin{proof}[Proof of Theorem \ref{t12}]
Keep the notations. Consider the $m$-fold fibre product
\[
f_{m}:X_{m}=X\times_{Y}\cdot\cdot\cdot\times_{Y}X\rightarrow Y.
\]
If we denote
\[
    L_{m}:=(p^{m}_{1}\otimes p^{m-1}_{1}p^{m}_{2}\otimes\cdot\cdot\cdot\otimes p^{1}_{2}p^{2}_{2}...p^{m}_{2})^{\ast}(L),
\]
by Lemma \ref{l41}, we have
\[
   (f_{\ast}(K_{X/Y}\otimes L))^{\otimes m}=(f_{m})_{\ast}(K_{X_{m}/Y}\otimes L_{m}).
\]
Moreover, Lemma \ref{l41} also implies that
\[
    \mathscr{I}(\varphi_{m})\subset(p^{m}_{1}\otimes p^{m-1}_{1}p^{m}_{2}\otimes p^{m-2}_{1}p^{m-1}_{2}p^{m}_{2}\otimes\cdot\cdot\cdot\otimes p^{1}_{2}p^{2}_{2}...p^{m}_{2})^{\ast}\mathscr{I}(\varphi),
\]
so we have
\[
\begin{split}
    &(f_{m})_{\ast}(K_{X_{m}/Y}\otimes L_{m}\otimes\mathscr{I}(\varphi_{m}))\\
    \subset&(f_{m})_{\ast}(K_{X_{m}/Y}\otimes L_{m}\otimes(p^{m}_{1}\otimes p^{m-1}_{1}p^{m}_{2}\otimes\cdot\cdot\cdot\otimes p^{1}_{2}p^{2}_{2}...p^{m}_{2})^{\ast}\mathscr{I}(\varphi))\\
    =&f_{\ast}(K_{X/Y}\otimes L\otimes\mathscr{I}(\varphi))^{\otimes m}.
\end{split}
\]
On the other hand, Lemma \ref{l42} says that the opposite direction of this inclusion holds, too. Thus, we actually have
\[
    \mathcal{E}^{\otimes m}:=f_{\ast}(K_{X/Y}\otimes L\otimes\mathscr{I}(\varphi))^{\otimes m}=(f_{m})_{\ast}(K_{X_{m}/Y}\otimes L_{m}\otimes\mathscr{I}(\varphi_{m}))
\]
for all positive integer $m$.

Then, we fix a very ample line bundle $A$ on $Y$. Let $A^{\prime}=\mathcal{O}_{Y}$ and let $H=K_{Y}\otimes A^{\otimes(\dim Y+1)}$. By way of motivation, imagine for the moment that one had a singular metric $\chi_{m}$ on $-K_{X_{m}/Y}\otimes L^{p}_{m}$ such that
\[
i\Theta_{-K_{X_{m}/Y}\otimes L^{p}_{m},\chi_{m}}\geqslant0\textrm{ and }\mathscr{I}(\tau_{m})=\mathcal{O}_{X_{m}}.
\]
Then applying Corollary \ref{c11} to the fibration $f_{m}$ as well as the direct image 
\[
(f_{m})_{\ast}(K_{X_{m}/Y}\otimes L_{m}\otimes\mathscr{I}(\varphi_{m})),
\] 
we deduce that the sheaf $\mathcal{E}^{\otimes m}\otimes H$ is generated by its global sections. In particular, $\mathcal{E}$ is weakly positive in the sense of Viehweg. While in reality the existence of $\chi_{m}$ may be too much to hope for, a simple observation is that 
\[
\tau_{m}:=(p^{m}_{1}+p^{m-1}_{1}p^{m}_{2}+p^{m-2}_{1}p^{m-1}_{2}p^{m}_{2}+...+ p^{1}_{2}p^{2}_{2}...p^{m}_{2})^{\ast}\tau
\]
is just as good. In fact, recall the proof of Theorem \ref{t13} (and the notations there), we only use the fact that $\mathscr{I}(\tau|_{F})=\mathcal{O}_{F}$. Returning to the situation of this theorem, the general fibre of $f_{m}:X_{m}\rightarrow Y$ is $F_{m}:=\underbrace{F\times\cdots\times F}_{m}$. Let $q_{i}:F_{m}\rightarrow F$ be the $i$-th projection. Obviously
\[
\mathscr{I}(\tau_{m}|_{F_{m}})=\mathscr{I}(\sum_{i} q^{\ast}_{i}(\tau|_{F}))=\mathcal{O}_{F_{m}},
\]
hence Corollary \ref{c11} still applies here.

Now the torsion-free coherent sheaf $S^{m}\mathcal{E}\otimes H$, being a quotient of $\mathcal{E}^{\otimes m}\otimes H$, is also globally generated. Consider $\pi:\mathbb{P}(\mathcal{E}^{\ast})\rightarrow Y$. Here $\mathbb{P}(\mathcal{E}^{\ast})$ refers to the projective space bundle \cite{Har77} associated to a coherent sheaf. Note that we have the surjective morphism
\[
    \pi^{\ast}\pi_{\ast}\mathcal{O}_{\mathcal{E}}(m)\cong\pi^{\ast}(S^{m}(\mathcal{E}))\rightarrow\mathcal{O}_{\mathcal{E}}(m),
\]
and we thus deduce that $\mathcal{O}_{\mathcal{E}}(m)\otimes\pi^{\ast}H$ is globally generated (hence nef) for every $m\geqslant1$. This implies that $\mathcal{O}_{\mathcal{E}}(1)$ is nef, that is, $\mathcal{E}$ is nef.
\end{proof}

In the end, we discuss several special cases of Theorem \ref{t12}. Firstly, we prove Theorem \ref{t11}.

\begin{proof}[Proof of Theorem \ref{t11}]
$m=0$ is trivial. $m=1$ is furnished by Griffiths \cite{Gri84} and Fujita--Kawamata \cite{Kaw82}. So we assume $m>1$ without loss of generality.

Note that if $\varphi=\{\varphi_{U}\}$ is the collection of metrics associated to 
\[
\mathscr{I}(f,\|K_{X/Y}\|),
\] 
then $m\varphi=\{m\varphi_{U}\}$ is the collection of metrics associated to $\mathscr{I}(f,\|K^{m}_{X/Y}\|)$. Obviously we have 
\[
(m-1)\varphi\preceq m\varphi.
\]

On the other hand, let $\psi$ be the metric on $K_{X/Y}$ such that $\mathscr{I}(\psi)=\mathscr{I}(\|K_{X/Y}\|)$ (see Sect.\ref{sec:asymptotic}), then $i\Theta_{K^{p(m-1)-1}_{X/Y},(p(m-1)-1)\psi}\geqslant0$ and 
\[
\mathscr{I}((p(m-1)-1)\psi)=\mathscr{I}((p(m-1)-1)\|K_{X/Y}\|)=\mathcal{O}_{X} 
\] 
for $m\leqslant k+1$.

Apply Theorem \ref{t12} with $L=K^{m-1}_{X/Y}$, and observe that $\mathscr{I}(f,(m-1)\|K_{X/Y}\|)$ is also trivial by Proposition \ref{p22}, (iv), we then obtain that $f_{\ast}(K^{m}_{X/Y})$ is nef for $m\leqslant k+1$.
\end{proof}

\begin{corollary}\label{c41}
Let $f:X\rightarrow Y$ be a smooth fibration between projective manifolds $X$ and $Y$. Let $L$ be a holomorphic line bundle on $X$ with $\kappa(K_{X/Y}\otimes L,f)\geqslant0$. Fix $p\gg0$ and divisible enough that computes $\mathscr{I}(f,\|L\|)$. Assume that $-K_{X/Y}$ is semi-ample, and there exists a (singular) metric $\tau$ on $L^{p}$ such that $i\Theta_{L^{p},\tau}\geqslant0$ and $\mathscr{I}(\tau)=\mathcal{O}_{X}$. Then 
\[
f_{\ast}(K_{X/Y}\otimes L\otimes\mathscr{I}(f,\|L\|))
\]
is nef.
\begin{proof}
Fix $q\gg0$ and divisible enough that computes both of $\mathscr{I}(f,\|K_{X/Y}\otimes L\|)$ and $\mathscr{I}(f,\|L\|)$. Furthermore, $|-K^{q}_{X/Y}|$ is base-point free. 

Let $\{\alpha_{i}\}$ be a basis of $H^{0}(X,(K_{X/Y}\otimes L)^{q})$ and let $\{\beta_{j}\}$ be a basis of $H^{0}(X,-K^{q}_{X/Y})$. Then all of $\alpha_{i}\otimes\beta_{j}$ form a linear subspace $\mathfrak{d}$ of $|L^{q}|$. In particular, the base-ideal $\mathfrak{a}$ of $\mathfrak{d}$ is contained in the base-ideal $\mathfrak{b}$ of $|L^{q}|$. Note that $\mathfrak{a}$ is also the base-ideal of $|(K_{X/Y}\otimes L)^{q}|$ since $|-K^{q}_{X/Y}|$ is base-point free. 

Let $\{\gamma_{k}\}$ be the generators of $\mathfrak{b}$. Then
\[
\log(\sum_{k}|\gamma_{k}|^{2})\preceq\log(\sum_{i,j}|\alpha_{i}\otimes\beta_{j}|^{2})
\]
due to the inclusion of the corresponding ideals. Therefore the second requirement of Theorem \ref{t12} is verified by definition. The first requirement is by assumption. Now the conclusion follows by applying Theorem \ref{t12}.
\end{proof}
\end{corollary}

\begin{corollary}\label{c42}
Let $f:X\rightarrow Y$ be a smooth fibration between projective manifolds $X$ and $Y$ with $\kappa(-K_{X/Y},f)\geqslant0$. Assume that $\mathscr{I}(k\|-K_{X/Y}\|)=\mathcal{O}_{X}$ for every $k\geqslant0$. Then
\[
f_{\ast}(-K^{m}_{X/Y})
\] 
is nef for any integer $m\geqslant -1$. 
\begin{proof}
The case where $m=-1$ is due to Griffiths \cite{Gri84} and Fujita--Kawamata \cite{Kaw82}, which is even valid without any extra assumption.

Now $m\geqslant0$. Since $\mathscr{I}(m\|-K_{X/Y}\|)=\mathcal{O}_{X}$, we have $\mathscr{I}(f,m\|-K_{X/Y}\|)=\mathcal{O}_{X}$ by Proposition \ref{p22}, (iv). Recall that there exists a metric $\psi$ on $-K_{X/Y}$ such that $i\Theta_{-K_{X/Y},\psi}\geqslant0$ and 
\[
\mathscr{I}(m\psi)=\mathscr{I}(m\|-K_{X/Y}\|).
\] 
So the requirement (1) of Theorem \ref{t12} is verified.

On the other hand, let $\varphi=\{\varphi_{U}\}$ be the collection of metrics associated to 
\[
\mathscr{I}(f,\|-K_{X/Y}\|).
\]  
Then
\[
\mathscr{I}(k\varphi)=\mathcal{O}_{X}
\]
for any $k$. It implies that $\varphi$ is actually smooth concerning the fact that $\varphi$ has algebraic singularities \cite{Dem12}. The second requirement now is obviously satisfied. This proves the conclusion for the cases when $k\geqslant0$ by Theorem \ref{t12}.

\end{proof}
\end{corollary}

\begin{acknowledgements}
The author wants to thank Prof. Jixiang Fu for his suggestion and encouragement.
\end{acknowledgements}

\end{document}